\documentclass[12pt,twoside,a4paper]{article}

\usepackage{amssymb}
\usepackage{listings}
\usepackage{url}

\usepackage{ifthen}

\newcounter{unit}
\renewcommand{\theunit}{\arabic{unit}}
\newcommand{\unitlabelstyle}{\bfseries}
\newcommand{\unitreferencestyle}{}
\newcommand{\unitname}{???}
\newcommand{\unitlabel}[1]{#1}
\newcommand{\unitclose}{}
\newcommand{\increaseunitcounter}[1][y]{%
	\ifthenelse{\equal{#1}{n}}{}{\refstepcounter{unit}}%
	\renewcommand{\unitlabel}[1]{%
		\mbox{%
			{\unitlabelstyle\unitname\ifthenelse{\equal{#1}{n}}{}{\ \theunit}}%
			\ifthenelse{\equal{##1}{}}{}{\unitreferencestyle\ (##1)}%
		}%
	}%
}%
\newcommand{\newunit}[5][y]{%
	\newenvironment{#2}[1][]{%
		\increaseunitcounter[#1]%
		\renewcommand{\unitname}{#3}%
		\renewcommand{\unitclose}{#4}%
		\begin{list}{}{%
			\renewcommand{\makelabel}{\unitlabel}%
		}\item[##1]#5%
	}{%
		\hspace*{\fill}\unitclose%
		\end{list}%
		\renewcommand{\unitname}{???}%
		\renewcommand{\unitlabel}[1]{#1}%
		\renewcommand{\unitclose}{}%
	}%
}

\newcommand{\notationname}{Notation}      \newcommand{\notationend}{\ensuremath{\bowtie}}       \newcommand{\notationstyle}{}
\newcommand{\definitionname}{Definition}  \newcommand{\definitionend}{\ensuremath{\circ}}       \newcommand{\definitionstyle}{}
\newcommand{\remarkname}{Remark}          \newcommand{\remarkend}{\ensuremath{\triangleleft}}   \newcommand{\remarkstyle}{}
\newcommand{\examplename}{Example}        \newcommand{\exampleend}{\ensuremath{\triangleright}} \newcommand{\examplestyle}{}
\newcommand{\conjecturename}{Conjecture}  \newcommand{\conjectureend}{\ensuremath{\diamond}}    \newcommand{\conjecturestyle}{\itshape}
\newcommand{\lemmaname}{Lemma}            \newcommand{\lemmaend}{\ensuremath{\diamond}}         \newcommand{\lemmastyle}{\itshape}
\newcommand{\theoremname}{Theorem}        \newcommand{\theoremend}{\ensuremath{\diamond}}       \newcommand{\theoremstyle}{\itshape}
\newcommand{\propositionname}{Proposition}\newcommand{\propositionend}{\ensuremath{\diamond}}   \newcommand{\propositionstyle}{\itshape}
\newcommand{\corollaryname}{Corollary}    \newcommand{\corollaryend}{\ensuremath{\diamond}}     \newcommand{\corollarystyle}{\itshape}
\newcommand{\proofname}{Proof}            \newcommand{\proofend}{\ensuremath{\square}}          \newcommand{\proofstyle}{}

\newunit[n]{notation}{\notationname}{\notationend}{\notationstyle}
\newunit[n]{definition}{\definitionname}{\definitionend}{\definitionstyle}
\newunit[n]{remark}{\remarkname}{\remarkend}{\remarkstyle}
\newunit{example}{\examplename}{\exampleend}{\examplestyle}
\newunit{conjecture}{\conjecturename}{\conjectureend}{\conjecturestyle}
\newunit{lemma}{\lemmaname}{\lemmaend}{\lemmastyle}
\newunit{theorem}{\theoremname}{\theoremend}{\theoremstyle}
\newunit{proposition}{\propositionname}{\propositionend}{\propositionstyle}
\newunit{corollary}{\corollaryname}{\corollaryend}{\corollarystyle}
\newunit[n]{proof}{\proofname}{\proofend}{\proofstyle}

\newcommand{\N}{\ensuremath{\mathbb{N}}}
\newcommand{\blog}[1][b]{\ensuremath{\log_{#1}}}
\renewcommand{\P}{\ensuremath{\theta}}
\newcommand{\floor}[1]{\ensuremath{\left\lfloor#1\right\rfloor}}

\lstdefinelanguage{GAP}
  {morekeywords={and,do,elif,else,end,fi,for,function,if,in,local,mod,not,od,%
		repeat,return,then,until,while,quit,QUIT,break,rec,continue},
  sensitive=true,
  morecomment=[l]{\#},
  morestring=[b]",
}

\title{On a curious property of 3435.}
\author{Daan van Berkel}

\begin{document}
	\maketitle
	\begin{abstract}
	Folklore tells us that there are no uninteresting natural numbers. But some 
	natural numbers are more interesting then others. In this article we will 
	explain why $3435$ is one of the more interesting natural numbers around.
	
	We will show that $3435$ is a \emph{Munchausen number} in base 10, and we 
	will explain what we mean by that. We will further show that for every base
	there are finitely many Munchausen numbers in that base.
\end{abstract}

	Folklore tells us that there are no uninteresting natural numbers. The argument
hinges on the following observation: \emph{Every subset of the natural numbers
is either empty, or has a smallest element}.

The argument usually goes something like this. If there would be any 
uninteresting natural numbers, the set $\mathcal{U}$ of all these uninteresting
natural numbers would have a smallest element, say $u \in \mathcal{U}$. 
But $u$ in it self has a very remarkable property. $u$ is the smallest 
uninteresting natural number, which is very interesting indeed. So 
$\mathcal{U}$, the set of all the uninteresting natural numbers, can not have a
smallest element, therefore $\mathcal{U}$ must be empty. In other words, all 
natural numbers are interesting.

Having established this result, exhibiting an interesting property of a specific
natural number is often left as an excercise for the reader. Take for example 
the integer $3435$. At first it does not seem that remarkable, until one 
stumbles upon the following identity.
\[
	3435 = 3^{3} + 4^{4} + 3^{3} + 5^{5}
\]
This coincidence is even more remarkable when one discovers that there is only 
one other natural number which shares this property with $3435$, namely
$1 = 1^{1}$.

In this article we will establish the claim made and generalize the result.

	\section*{Munchausen Number}
Through out the article we will use the following notation. $b \in \N$ will 
denotate a base and therefore the inequility $b \ge 2$ will hold throughout the
article. For every natural number $n \in \N$, the \emph{base $b$ representation
of $n$} will be denoted by $[c_{m-1}, c_{m-2}, \ldots, c_{0}]_{b}$, 
so $0 \le c_{i} < b$ for all $i \in \{0,1,\ldots,m-1\}$ and 
$n = \sum_{i=0}^{m-1} c_{i}b^{i}$.
Furtheremore, we define a function $\P_{b} : \N \rightarrow \N : n \mapsto 
\sum_{i=0}^{m-1} c_{i}^{c_{i}}$, where $n = [c_{m-1},c_{m-2},\ldots,c_{0}]_{b}$.
We will further adopt the convention that $0^{0} = 1$, in accordance with 
$1^{0} = 1$, $2^{0} = 1$ etcetera.

\begin{definition}
	An integer $n \in \N$ is called a \emph{Munchausen number in base $b$} if 
	and only if $n = \P_{b}(n)$.
\end{definition}

So by the equality in the introduction we know that $3435$ is a Munchausen
number in base $10$. 

\begin{remark}
	A related concept to Munchausen number is that of Narcissistic number. 
	(See for example \cite{pickover}, \cite{wikipedia:narcissistic_number} and
	\cite{wolfram:narcissistic_number}.)
	
	The reason for picking the name Munchausen number stems from the visual of
	raising oneself, a feat demonstrated by the famous Baron von Munchausen 
	(\cite{wikipedia:munchausen}). Andrew Baxter remarked that the Baron is a 
	narcissistic man indeed, so I think the name is aptly chosen.
\end{remark}

The following two lemmas will be used to proof the main 
result of this article: for every base $b \in \N$ there are only finitely many 
Munchausen numbers in base $b$.

\begin{lemma}
	For all $n \in \N$: $\P_{b}(n) \le (\blog(n) + 1)(b-1)^{b-1}$.
\end{lemma}

\begin{proof}
	Notice that the function $x \mapsto x^{x}$ is strictly increasing if 
	$x \ge \frac{1}{e}$. This can be seen from the derivative of $x^{x}$ which 
	is $x^x(\log(x) + 1)$. This last expression is clearly positive for 
	$x > \frac{1}{e}$.
	Together with the definition of $0^{0} = 1$, we see that $x^{x}$ is 
	increasing for all the nonnegative integers.
	
	For all $n \in \N$ with $n = [c_{m-1}, c_{m-2}, \ldots, c_{0}]_{b}$ we have 
	the ineqalities $0 \le c_{i} \le b-1$ for all $i$ within $0 \le i < m$.	\\
	So $\P_{b}(n) = \sum_{i=0}^{m-1} c_{i}^{c_{i}} \le 
	\sum_{i=0}^{m-1} (b-1)^{b-1} = m \times (b-1)^{b-1}$.
	
	Now, the number of digits in the base $b$ represantation of $n$ equals 
	$\floor{\blog(n) + 1}$. In other words $m := \floor{\blog(n) + 1} 
	\le \blog(n) + 1$.
	
	So $\P_{b}(n) \le (\blog(n) + 1)(b-1)^{b-1}$
\end{proof}

\begin{lemma}
	If $n \in \N$ and $n > 2b^{b}$ then $\frac{n}{\blog(n) + 1}	> (b-1)^{b-1}$.
\end{lemma}

\begin{proof}
	Let $n \in \N$ such that $n > 2b^{b}$. Notice that 
	$x \mapsto \frac{x}{\blog(x)}$ is strictly increasing if $x > e$. To see
	this notice that the derivative of $\frac{x}{\blog{x}}$ is 
	$\log(b)\frac{\log(x) - 1}{\log^2(x)}$ which is positive for $x > e$.
	Furthermore $\blog(2) + 1 \le 2 \le b = b\blog(b)$. 
	
	Now, because $n > 2b^{b} > e$, from the following chain of ineqalities:
	\[
		\frac{n}{\blog(n) + 1} > \frac{2b^{b}}{b\blog(b) + \blog(2) + 1} \ge 
		\frac{2b^{b}}{2b\blog(b)} = b^{b-1} > (b-1)^{b-1}
	\]
	we can deduce that $\frac{n}{\blog(n) + 1} > (b-1)^{b-1}$
\end{proof}

With both lemma's in place we can present without further ado the main result of
this article.

\begin{proposition}
	For every base $b \in \N$ with $b \ge 2$: there are only finitely many 
	Munchausen numbers in base $b$.
\end{proposition}

\begin{proof}
	By the preceding lemma's we have, for all $n \in \N$ with $n > 2b^{b}$: 
	$n > (\blog(n) + 1)(b-1)^{b-1} \ge \P_{b}(n)$.
	
	So, in order for $n$ to equal $\P_{b}(n)$, $n$ must be less then or equal to 
	$2b^{b}$. This proves that there are only finitely many Munchausen numbers
	in base $b$.
\end{proof}

	\section*{Exhaustive Search}

The proposition in the preceding section tells use that for every base 
$b \in \N$, Munchausen numbers in that base only occur within the interval 
$[1,2b^{b}]$. This makes it possible to exhaustively search for Munchausen 
numbers in each base.

Figure \ref{figure:munchausen} lists all the Munchausen numbers in the bases 2
through 10. So for example in base $4$, $29$ and $55$ are the only non-trivial
Munchausen numbers. Furthermore, the base $4$ representation of $29$ and $55$
have a striking resemblance. For $29 = [1,3,1]_{4} = 1^{1} + 3^{3} + 1^{1}$ and
$55 = [3,1,3]_{4} = 3^{3} + 1^{1} + 3^{3}$.

\begin{figure}[th]
	\begin{center}
		\caption{Munchausen numbers in base 2 through 10.}
		\label{figure:munchausen}
		\begin{tabular}{|c|l|l|}
			\hline
			Base & Munchausen Numbers & Representation \\
			\hline
			2  & 1, 2                 & {\small$[1]_{2}$, $[1,0]_{2}$} \\
			3  & 1, 5, 8              & {\small$[1]_{3}$, $[1,2]_{3}$, $[2,2]_{3}$} \\
			4  & 1, 29, 55            & {\small$[1]_{4}$, $[1,3,1]_{4}$, $[3,1,3]_{4}$} \\
			5  & 1                    & {\small$[1]_{5}$} \\
			6  & 1, 3164, 3416        & {\small$[1]_{6}$, $[2,2,3,5,2]_{6}$, $[2,3,4,5,2]_{6}$} \\
			7  & 1, 3665              & {\small$[1]_{7}$, $[1,3,4,5,4]_{7}$} \\
			8  & 1                    & {\small$[1]_{8}$} \\
			9  & 1, 28, 96446, 923362 & {\small$[1]_{9}$, $[3,1]_{9}$, $[1,5,6,2,6,2]_{9}$, $[1,6,5,6,5,4,7]_{9}$} \\
			10 & 1, 3435              & {\small$[1]_{10}$, $[3,4,3,5]_{10}$} \\
			\hline
		\end{tabular}
	\end{center}
\end{figure}

The sequence of Munchausen numbers is listed as sequence A166623 at the OEIS. 
(See \cite{oeis:munchausen}. For the related sequence of Narcissistic numbers see
\cite{oeis:narcissistic})

The code in listing \ref{code:munchausen} is used to produce the numbers in
figure \ref{figure:munchausen}. There are two utility functions. These are 
\lstinline!munchausen! and \lstinline!next!. \lstinline!munchausen! calculates 
$\P_{b}(n)$ given a base $b$ representation of $n$. \lstinline!next! returns the
base $b$ representation of $n+1$ given a base $b$ representation of $n$.

I would like to conclude this article with a question my wife asked me while I
was writing this: ``But what about $20082009$?''

\begin{lstlisting}[
	caption={GAP code finding Munchausen numbers},
	label=code:munchausen,
	basicstyle=\scriptsize
]
next := function(coefficients, b)
	local i;
	coefficients[1] := coefficients[1] + 1;
	i := 1;
	while coefficients[i] = b do
		coefficients[i] := 0;
		i := i + 1;
		if (i <= Length(coefficients)) then
			coefficients[i] := coefficients[i] + 1;
		else
			Add(coefficients, 1);
		fi;
	od;
	return coefficients;
end;

munchausen := function(coefficients)
	local sum, coefficient;
	sum := 0;
	for coefficient in coefficients do
		sum := sum + coefficient^coefficient;
	od;
	return sum;
end;

for b in [2..10] do
	max := 2*b^b;
	n := 1;	coefficients := [1];
	while n <= max do
		sum := munchausen(coefficients);
	
		if (n = sum) then
			Print(n, "\n");
		fi;
		
		n := n + 1; 
		coefficients := next(coefficients, b);
	od;
od;
\end{lstlisting}

	\bibliographystyle{unsrt}
\bibliography{3435}

\end{document}